\documentclass{article}

\usepackage{amsmath,amssymb,amsthm}

\newtheorem{thm}{Theorem}
\newtheorem{cor}[thm]{Corollary}

\newtheorem{prop}[thm]{Proposition}
\newtheorem{example}[thm]{Example}

\theoremstyle{remark}
\newtheorem{rem}[thm]{Remark}

\theoremstyle{definition}

\begin{document}

\title{Integral Inequalities and their Applications\\
to the Calculus of Variations on Time Scales\thanks{To 
appear in \emph{Mathematical Inequalities \& Applications} 
({\tt http://mia.ele-math.com}).}}

\author{Martin J. Bohner\thanks{Email: \texttt{bohner@mst.edu}.}\\
Department of Mathematics and Statistics\\
Missouri University of Science and Technology\\
Rolla, Missouri 65409-0020, USA
\and
Rui A. C. Ferreira\thanks{Email: \texttt{ruiacferreira@ua.pt}.
Supported by \emph{The Portuguese Foundation for Science and Technology}
(FCT) through the PhD fellowship SFRH/BD/39816/2007.}\\
Department of Engineering and Natural Sciences\\
Lusophone University of Humanities and Tecnologies\\
1749-024 Lisbon, Portugal
\and
Delfim F. M. Torres\thanks{Email: \texttt{delfim@ua.pt}.
Supported by FCT through the R\&D unit \emph{Centre for Research
on Optimization and Control} (CEOC),
cofinanced by FEDER/POCI 2010.}\\
Department of Mathematics\\
University of Aveiro\\
3810-193 Aveiro, Portugal
}

\date{{\footnotesize (Submitted: 04.11.2008; Accepted: 14.01.2010)}}

\maketitle


\begin{abstract}
We discuss the use of inequalities to obtain
the solution of certain variational problems
on time scales.

\bigskip


\bigskip


\end{abstract}


\section{Introduction}

A time scale, denoted by $\mathbb{T}$, is a nonempty
closed subset of the real numbers. The calculus on time scales
is a relatively new area that unifies the difference
and differential calculus, which are obtained by choosing
$\mathbb{T}=\mathbb{Z}$ or $\mathbb{T}=\mathbb{R}$, respectively.
The subject was initiated by S.~Hilger in the nineties of the XX century
\cite{Hilger90,Hilger97}, and is now under strong
current research in many different fields in which dynamic processes
can be described with discrete, continuous, or hybrid models.
For concepts and preliminary results on time scales,
we refer the reader to \cite{livro,advance}.

In this paper we start by proving some integral inequalities on time scales
involving convex functions (see Section~\ref{ineq}). These are then applied
in Section~\ref{sec:app:CV} to solve some classes of variational problems
on time scales. A simple illustrative example is given in Section~\ref{sec:ex}.
The method proposed here is direct, in the sense that it permits
to find directly the optimal solution instead of using variational arguments
and go through the usual procedure of solving the associated
delta or nabla Euler--Lagrange equations
\cite{Atici:et:al,Bohner:CV,remarks,Hilscher:Zeidan:2004}.
This is particularly useful since even simple classes of problems
of the calculus of variations on time scales lead to dynamic
Euler--Lagrange equations for which methods
to compute explicit solutions are not known.
A second advantage of the method is that it provides
directly an optimal solution, while the variational method on time scales
initiated in \cite{Bohner:CV} and further developed in
\cite{Bartosiewicz:Torres,Ferreira:Torres:HO,iso:ts,Malinowska:Torres}
is based on necessary optimality conditions, being necessary further analysis
in order to conclude if the candidate is a local minimizer, a local maximizer,
or just a saddle (see \cite{Bohner:CV} for second order necessary
and sufficient conditions). Finally, while all the previous methods
of the calculus of variations on time scales only establish local optimality,
here we provide global solutions.

The use of inequalities to solve certain classes
of optimal control problems is an old idea with a rich history
\cite{wsc,otto,Rui:Rachid:Delfim,H:L:1932,inequalities,Sbordone}.
We hope that the present study will be the beginning
of a class of direct methods for optimal control problems
on time scales, to be investigated with the help of dynamic inequalities
--- see \cite{inesurvey,Bohner:Matthews,Rui:Rachid:Delfim,GronTS,SRT,SA:T,wong}
and references therein.


\section{Integral Inequalities on Time Scales}
\label{ineq}

The first theorem is a generalization to time scales
of the well-known
\emph{Jensen inequality}. It can be found in \cite{SRT,wong}.

\begin{thm}[Generalized Jensen's inequality \cite{SRT,wong}]
\label{thm1} Let $a,b\in\mathbb{T}$ and $c,d\in\mathbb{R}$.
Suppose $f:[a,b]^\kappa_{\mathbb{T}}\rightarrow(c,d)$ is
rd-continuous and $F:(c,d)\rightarrow\mathbb{R}$ is convex.
Moreover, let $h:[a,b]^\kappa_{\mathbb{T}}\rightarrow\mathbb{R}$
be rd-continuous with
$$\int_a^b|h(t)|\Delta t>0.$$
Then,
\begin{equation}
\label{in1} \frac{\int_a^b |h(t)|F(f(t))\Delta
t}{\int_a^b|h(t)|\Delta t}\geq
F\left(\frac{\int_a^b|h(t)|f(t)\Delta t}{\int_a^b|h(t)|\Delta
t}\right).
\end{equation}
\end{thm}

\begin{prop}
\label{rem0}
If $F$ in Theorem~\ref{thm1} is strictly convex
and $h(t)\neq 0$ for all
$t\in[a,b]^\kappa_\mathbb{T}$, then equality in \eqref{in1} holds
if and only if $f$ is constant.
\end{prop}
\begin{proof}
Consider $x_0\in(c,d)$ defined by
$$x_0=\frac{\int_a^b|h(t)|f(t)\Delta t}{\int_a^b|h(t)|\Delta
t}.$$ From the definition of strict convexity,
there exists $m\in\mathbb{R}$ such that
$$F(x)-F(x_0)>m(x-x_0)$$
for all $x\in (c,d)\backslash\{x_0\}$. Assume
$f$ is not constant. Then, $f(t_0)\neq x_0$ for some
$t_0\in[a,b]^\kappa_{\mathbb{T}}$. We split the proof in two
cases. (i) Assume that $t_0$ is right-dense. Then, since $f$ is
rd-continuous, we have that $f(t)\neq x_0$ on
$[t_0,t_0+\delta)_\mathbb{T}$ for some $\delta>0$. Hence,
\begin{align*}
\int_a^b |h(t)|F(f(t))\Delta t-\int_a^b |h(t)|\Delta t
F(x_0)&=\int_a^b |h(t)|[F(f(t))-F(x_0)]\Delta t\\
&>m\int_a^b |h(t)|[f(t)-x_0]\Delta t\\ &=0 \, .
\end{align*}
(ii) Assume now that $t_0$ is right-scattered. Then (note that
$\int_{t_0}^{\sigma(t_0)} f(t)\Delta t=\mu(t_0)f(t_0)$),
\begin{equation*}
\begin{split}
\int_a^b & |h(t)|F(f(t))\Delta t-\int_a^b |h(t)|\Delta t
F(x_0)\\
&= \int_a^b |h(t)|[F(f(t))-F(x_0)]\Delta t\\
&= \int_a^{t_0}
|h(t)|[F(f(t))-F(x_0)]\Delta t+\int_{t_0}^{\sigma(t_0)}
|h(t)|[F(f(t))-F(x_0)]\Delta t\\
& \qquad +\int_{\sigma(t_0)}^{b}
|h(t)|[F(f(t))-F(x_0)]\Delta t\\
&> \int_a^{t_0} |h(t)|[F(f(t))-F(x_0)]\Delta
t+m\int_{t_0}^{\sigma(t_0)} |h(t)|[f(t)-x_0]\Delta t\\
& \qquad +\int_{\sigma(t_0)}^{b} |h(t)|[F(f(t))-F(x_0)]\Delta t\\
&\geq m\left\{\int_{a}^{t_0} |h(t)|[f(t)-x_0]\Delta
t+\int_{t_0}^{\sigma(t_0)} |h(t)|[f(t)-x_0]\Delta t\right.\\
& \qquad \left.+\int_{\sigma(t_0)}^{b} |h(t)|[f(t)-x_0]\Delta t\right\}\\
&= m\int_a^b|h(t)|[f(t)-x_0]\Delta t =0 \, .
\end{split}
\end{equation*}
Finally, if $f$ is constant, it is obvious that the equality in
\eqref{in1} holds.
\end{proof}

\begin{rem}
If $F$ in Theorem~\ref{thm1} is a concave function, then the
inequality sign in \eqref{in1} must be reversed. Obviously,
Proposition~\ref{rem0} remains true if we let $F$
be strictly concave.
\end{rem}

Before proceeding, we state two particular cases of
Theorem~\ref{thm1}.

\begin{cor}
Let $a,b,c,d\in\mathbb{R}$. Suppose $f:[a,b]\rightarrow(c,d)$ is
continuous and $F:(c,d)\rightarrow\mathbb{R}$ is convex. Moreover,
let $h:[a,b]\rightarrow\mathbb{R}$ be continuous with
$$\int_a^b|h(t)|dt>0.$$
Then,
\begin{equation*}
\frac{\int_a^b |h(t)|F(f(t))dt}{\int_a^b|h(t)|dt}\geq
F\left(\frac{\int_a^b|h(t)|f(t)dt}{\int_a^b|h(t)|dt}\right).
\end{equation*}
\end{cor}
\begin{proof}
Choose $\mathbb{T}=\mathbb{R}$ in Theorem~\ref{thm1}.
\end{proof}

\begin{cor}
Let $a=q^n$ and $b=q^m$ for some $n,m\in\mathbb{N}_0$ with $n<m$.
Define $f$ and $h$ on $[q^n,q^{m-1}]_{q^{\mathbb{N}_0}}$ and
assume $F:(c,d)\rightarrow\mathbb{R}$ is convex, where
$(c,d)\supset[f(q^n),f(q^{m-1})]_{q^{\mathbb{N}_0}}$. If
$$\sum_{k=m}^{n-1}q^k(q-1)|h(q^k)|>0,$$
then
\begin{equation*}
\frac{\sum_{k=m}^{n-1}q^k
|h(q^k)|F(f(q^k))}{\sum_{k=m}^{n-1}q^k|h(q^k)|}\geq
F\left(\frac{\sum_{k=m}^{n-1}q^k|h(q^k)|f(q^k)}{\sum_{k=m}^{n-1}q^k|h(q^k)|}\right).
\end{equation*}
\end{cor}
\begin{proof}
Choose $\mathbb{T}=q^{\mathbb{N}_0}=\{q^{k}:k\in\mathbb{N}_0\}$,
$q>1$, in Theorem~\ref{thm1}.
\end{proof}

Jensen's inequality \eqref{in0} is
proved in \cite[Theorem~4.1]{inesurvey}.

\begin{thm}
\label{thm0} Let $a,b\in\mathbb{T}$ and $c,d\in\mathbb{R}$.
Suppose $f:[a,b]^\kappa_{\mathbb{T}}\rightarrow(c,d)$ is
rd-continuous and $F:(c,d)\rightarrow\mathbb{R}$ is convex
(resp., concave). Then,
\begin{equation}
\label{in0} \frac{\int_a^b F(f(t))\Delta t}{b-a}\geq
F\left(\frac{\int_a^bf(t)\Delta t}{b-a}\right)
\end{equation}
(resp., the reverse inequality). Moreover, if $F$ is strictly
convex or strictly concave, then equality in \eqref{in0} holds if and
only if $f$ is constant.
\end{thm}
\begin{proof}
This is a particular case of Theorem~\ref{thm1}
and Proposition~\ref{rem0}
with $h(t)=1$ for all $t\in[a,b]^\kappa_{\mathbb{T}}$.
\end{proof}

We now state and prove some consequences of Theorem~\ref{thm0}.

\begin{cor}
\label{cor0} Let $a,b\in\mathbb{T}$ and $c,d\in\mathbb{R}$.
Suppose $f:[a,b]^\kappa_{\mathbb{T}}\rightarrow(c,d)$ is
rd-continuous and $F:(c,d)\rightarrow\mathbb{R}$ is such that
$F''\geq 0$ (resp., $F''\leq 0$). Then,
\begin{equation}
\label{in2} \frac{\int_a^b F(f(t))\Delta t}{b-a}\geq
F\left(\frac{\int_a^bf(t)\Delta t}{b-a}\right)
\end{equation}
(resp., the reverse inequality). Furthermore, if $F''>0$ or
$F''<0$, equality in \eqref{in2} holds if and only if $f$ is
constant.
\end{cor}
\begin{proof}
This follows immediately from Theorem \ref{thm0} and the facts
that a function $F$ with $F''\geq 0$ (resp., $F''\leq 0$) is
convex (resp., concave) and with $F''>0$ (resp., $F''\leq 0$) is
strictly convex (resp., strictly concave).
\end{proof}

\begin{cor}
Let $a,b\in\mathbb{T}$ and $c,d\in\mathbb{R}$. Suppose
$f:[a,b]^\kappa_{\mathbb{T}}\rightarrow(c,d)$ is rd-continuous and
$\varphi,\psi:(c,d)\rightarrow\mathbb{R}$ are continuous functions such
that $\varphi^{-1}$ exists, $\psi$ is strictly increasing, and
$\psi\circ\varphi^{-1}$ is convex (resp., concave) on
\textup{Im}($\varphi$). Then,
\begin{equation*}
\psi^{-1}\left(\frac{\int_a^b \psi(f(t))\Delta t}{b-a}\right)\geq
\varphi^{-1}\left(\frac{\int_a^b \varphi(f(t))\Delta
t}{b-a}\right)
\end{equation*}
(resp., the reverse inequality). Furthermore, if
$\psi\circ\varphi^{-1}$ is strictly convex or strictly concave,
the equality holds if and only if $f$ is constant.
\end{cor}
\begin{proof}
Since $\varphi$ is continuous and $\varphi\circ f$ is rd-continuous,
it follows from Theorem~\ref{thm0} with $f=\varphi\circ f$ and
$F=\psi\circ\varphi^{-1}$ that
$$\frac{\int_a^b (\psi\circ\varphi^{-1})((\varphi\circ
f)(t))\Delta t}{b-a}\geq
(\psi\circ\varphi^{-1})\left(\frac{\int_a^b (\varphi\circ
f)(t)\Delta t}{b-a}\right).$$ Since $\psi$ is strictly increasing,
we obtain
$$\psi^{-1}\left(\frac{\int_a^b \psi(f(t))\Delta
t}{b-a}\right)\geq
\varphi^{-1}\left(\frac{\int_a^b \varphi(f(t))\Delta
t}{b-a}\right).$$ Finally, when $\psi\circ\varphi^{-1}$ is
strictly convex, the equality holds if and only if $\varphi\circ
f$ is constant, or equivalently (since $\varphi$ is invertible),
$f$ is constant. The case when $\psi\circ\varphi^{-1}$ is concave
is treated analogously.
\end{proof}

\begin{cor}
\label{cor1} Assume
$f:[a,b]^\kappa_{\mathbb{T}}\rightarrow\mathbb{R}$ is
rd-continuous and positive. If $\alpha<0$ or $\alpha>1$, then
\begin{equation*}
\int_a^b (f(t))^\alpha\Delta t\geq(b-a)^{1-\alpha}\left(\int_a^b
f(t)\Delta t\right)^\alpha \, .
\end{equation*}
If $0<\alpha<1$, then
\begin{equation*}
\int_a^b (f(t))^\alpha\Delta t\leq(b-a)^{1-\alpha}\left(\int_a^b
f(t)\Delta t\right)^\alpha \, .
\end{equation*}
Furthermore, in both cases equality holds if and only if $f$ is
constant.
\end{cor}
\begin{proof}
Define $F(x)=x^\alpha,\ x>0$. Then
$$F''(x)=\alpha(\alpha-1)x^{\alpha-2},\quad x>0.$$
Hence, when $\alpha<0$ or $\alpha>1$, $F''>0$, \textrm{i.e.}, $F$ is
strictly convex. When $0<\alpha<1$, $F''<0$, \textrm{i.e.}, $F$ is strictly
concave. Applying Corollary \ref{cor0} with this function $F$, we
obtain the above inequalities with equality if and only if $f$ is
constant.
\end{proof}

\begin{cor}
Assume $f:[a,b]^\kappa_{\mathbb{T}}\rightarrow\mathbb{R}$ is
rd-continuous and positive. If $\alpha<-1$ or $\alpha>0$, then
\begin{equation*}
\left(\int_a^b \frac{1}{f(t)}\Delta t\right)^\alpha\int_a^b
(f(t))^\alpha\Delta t\geq(b-a)^{1+\alpha} \, .
\end{equation*}
If $-1<\alpha<0$, then
\begin{equation*}
\left(\int_a^b \frac{1}{f(t)}\Delta t\right)^\alpha\int_a^b
(f(t))^\alpha\Delta t\leq(b-a)^{1+\alpha}.
\end{equation*}
Furthermore, in both cases the equality holds if and only if $f$
is constant.
\end{cor}
\begin{proof}
This follows from Corollary \ref{cor1} by replacing $f$ by $1/f$
and $\alpha$ by $-\alpha$.
\end{proof}

\begin{cor}\label{cor3}
If $f:[a,b]^\kappa_{\mathbb{T}}\rightarrow\mathbb{R}$ is
rd-continuous, then
\begin{equation}\label{in3}
\int_a^b e^{f(t)}\Delta
t\geq(b-a)e^{\frac{1}{b-a}\int_a^bf(t)\Delta t} \, .
\end{equation}
Moreover, equality in \eqref{in3} holds if and only if $f$ is constant.
\end{cor}
\begin{proof}
Choose $F(x)=e^x$, $x\in\mathbb{R}$, in Corollary~\ref{cor0}.
\end{proof}

\begin{cor}\label{cor4}
If $f:[a,b]^\kappa_{\mathbb{T}}\rightarrow\mathbb{R}$ is
rd-continuous and positive, then
\begin{equation}\label{in4}
\int_a^b \ln(f(t))\Delta t\leq(b-a)\ln\left({\frac{1}{b-a}\int_a^b
f(t)\Delta t}\right) \, .
\end{equation}
Moreover, equality in \eqref{in4} holds if and only if $f$ is constant.
\end{cor}
\begin{proof}
Let $F(x)=\ln(x)$, $x>0$, in Corollary~\ref{cor0}.
\end{proof}

\begin{cor}\label{cor5}
If $f:[a,b]^\kappa_{\mathbb{T}}\rightarrow\mathbb{R}$ is
rd-continuous and positive, then
\begin{equation}
\label{eq:cor5}
\int_a^b f(t)\ln(f(t))\Delta t\geq\int_a^b f(t)\Delta
t\ln\left({\frac{1}{b-a}\int_a^b f(t)\Delta t}\right) \, .
\end{equation}
Moreover, equality in \eqref{eq:cor5} holds
if and only if $f$ is constant.
\end{cor}
\begin{proof}
Let $F(x)=x\ln(x)$, $x>0$. Then, $F''(x)=1/x$, \textrm{i.e.},
$F''(x)>0$ for all $x>0$. By Corollary \ref{cor0}, we get
$$\frac{1}{b-a}\int_a^b f(t)\ln(f(t))\Delta
t\geq\frac{1}{b-a}\int_a^b f(t)\Delta
t\ln\left({\frac{1}{b-a}\int_a^b f(t)\Delta t}\right),$$ and the
result follows.
\end{proof}


\section{Applications to the Calculus of Variations}
\label{sec:app:CV}

We now show how the results of Section~\ref{ineq} can
be applied to determine the minimum or maximum of
problems of calculus of variations and optimal control on time scales.

\begin{thm}\label{thm3}
Let $\mathbb{T}$ be a time scale, $a$, $b\in\mathbb{T}$
with $a < b$, and $\varphi:\mathbb{R}\rightarrow\mathbb{R}$
be a positive and continuous function. Consider the functional
$${\mathcal F}(y(\cdot))=\int_a^b
\left[\left\{\int_0^1\varphi(y(t)+h\mu(t)y^\Delta(t))dh\right\}
y^\Delta(t)\right]^\alpha\Delta t \, , \quad
\alpha\in\mathbb{R}\backslash\{0,1\} \, ,$$
defined on all $C_{\textup{rd}}^1$-functions
$y:[a,b]_{\mathbb{T}}\rightarrow\mathbb{R}$ satisfying
$y^\Delta(t)>0$ on $[a,b]^\kappa_{\mathbb{T}}$, $y(a)=0$, and
$y(b)=B$. Let $G(x)=\int_0^x\varphi(s)ds$, $x\geq 0$,
and let $G^{-1}$ denote its inverse. Let
\begin{equation}
\label{eq0}
C=\frac{\int_0^B \varphi(s)ds}{b-a}.
\end{equation}
\begin{description}
\item[(i)] If $\alpha<0$ or $\alpha>1$,
then the minimum of ${\mathcal F}$ occurs when
$$y(t)=G^{-1}(C(t-a)),\quad t\in[a,b]_\mathbb{T},$$
and ${\mathcal F}_{\min}=(b-a)C^\alpha$.
\item[(ii)] If $0<\alpha<1$,
then the maximum of ${\mathcal F}$ occurs when
$$y(t)=G^{-1}(C(t-a)),\quad t\in[a,b]_\mathbb{T},$$
and ${\mathcal F}_{\max}=(b-a)C^\alpha$.
\end{description}
\end{thm}

\begin{rem}
Since $\varphi$ is continuous and positive, $G$ and $G^{-1}$ are
well defined.
\end{rem}
\begin{rem}
In cases $\alpha=0$ or $\alpha=1$ there is nothing to minimize
or maximize, \textrm{i.e.}, the problem of extremizing
${\mathcal F}(y(\cdot))$
is trivial. Indeed, if $\alpha=0$, then
${\mathcal F}(y(\cdot))=b-a$; if $\alpha=1$,
then it follows from \cite[Theorem~1.90]{livro} that
\begin{align*}
{\mathcal F}(y(\cdot))&=\int_a^b
\left\{\int_0^1\varphi(y(t)+h\mu(t)y^\Delta(t))dh\right\}
y^\Delta(t)\Delta t\\ &=\int_a^b (G\circ
y)^\Delta(t)\Delta t\\ &=G(B) \, .
\end{align*}
In both cases ${\mathcal F}$ is a constant
and does not depend on the function $y$.
\end{rem}

\begin{proof}[Proof of Theorem~\ref{thm3}]
Suppose that $\alpha<0$ or $\alpha>1$.
Using Corollary~\ref{cor1}, we can write
\begin{equation*}
\begin{split}
{\mathcal F}(y(\cdot)) &\geq
(b-a)^{1-\alpha}\left[\int_a^b\left\{\int_0^1\varphi(y(t)+h\mu(t)y^\Delta(t))dh\right\}
y^\Delta(t)\Delta
t\right]^\alpha\\
&=(b-a)^{1-\alpha}(G(y(b))-G(y(a)))^\alpha \, ,
\end{split}
\end{equation*}
where the equality holds if and only if
$$\left\{\int_0^1\varphi(y(t)+h\mu(t)y^\Delta(t))dh\right\}
y^\Delta(t)=c \quad \mbox{for some}\ c\in\mathbb{R},\quad
t\in[a,b]^\kappa_\mathbb{T}.$$ Using \cite[Theorem~1.90]{livro},
we arrive at
$$(G\circ y)^\Delta(t)=c.$$
Delta integrating from $a$ to $t$ yields (note that $y(a)=0$ and
$G(0)=0$)
$$G(y(t))=c(t-a),$$
from which we get
$$y(t)=G^{-1}(c(t-a)).$$
The value of $c$ is obtained
using the boundary condition $y(b)=B$:
$$c=\frac{G(B)}{b-a}=\frac{\int_0^B \varphi(s)ds}{b-a}=C,$$
with $C$ as in \eqref{eq0}. Finally, in this case
$${\mathcal F}_{\min}=\int_a^b C^\alpha\Delta t=(b-a)C^\alpha.$$
The proof of the second part of the theorem is done analogously
using the second part of Corollary~\ref{cor1}.
\end{proof}

\begin{rem}
We note that the optimal solution found in the proof of the
previous theorem satisfies $y^\Delta>0$. Indeed,
\begin{align*}
y^\Delta(t)&=\left(G^{-1}(C(t-a))\right)^\Delta\\
&=\int_0^1 (G^{-1})'[C(t-a)+h\mu(t)C]dh\ C\\ &>0,
\end{align*}
because $C>0$ and $(G^{-1})'(G(x))=\frac{1}{\varphi(x)}>0$
for all $x\geq 0$.
\end{rem}

\begin{thm}
Let $\varphi:[a,b]^\kappa_\mathbb{T}\rightarrow\mathbb{R}$ be a
positive and rd-continuous function. Then, among all
$C_{\textup{rd}}^1$-functions
$y:[a,b]_{\mathbb{T}}\rightarrow\mathbb{R}$ with $y(a)=0$ and
$y(b)=B$, the functional
$${\mathcal F}(y(\cdot))
=\int_a^b\varphi(t)e^{y^\Delta(t)}\Delta t$$
has minimum value ${\mathcal F}_{\min}=(b-a)e^C$
attained when
$$y(t)=-\int_a^t\ln(\varphi(s))\Delta s+C(t-a),\quad
t\in[a,b]_\mathbb{T},$$
where
\begin{equation}
\label{eq1}
C=\frac{\int_a^b\ln(\varphi(t))\Delta
t+B}{b-a} \, .
\end{equation}
\end{thm}
\begin{proof}
By Corollary \ref{cor3},
\begin{multline*}
{\mathcal F}(y(\cdot))=\int_a^b e^{\ln(\varphi(t))+y^\Delta(t)}\Delta t\\
\geq(b-a)e^{\frac{1}{b-a}\int_a^b
[\ln(\varphi(t))+y^\Delta(t)]\Delta
t}=(b-a)e^{\frac{1}{b-a}\left[\int_a^b \ln(\varphi(t))\Delta
t+B\right]} \, ,
\end{multline*}
with ${\mathcal F}(y(\cdot))=(b-a)e^{\frac{1}{b-a}\left[\int_a^b \ln(\varphi(t))
\Delta t+B\right]}$ if and only if
\begin{equation}
\label{eq:prf:d}
\ln(\varphi(t))+y^\Delta(t)=c \quad \mbox{for some}\
c\in\mathbb{R},\quad
t\in[a,b]^\kappa_\mathbb{T} \, .
\end{equation}
Integrating \eqref{eq:prf:d} from $a$ to
$t$ (note that $y(a)=0$) gives
$$
y(t)=-\int_a^t\ln(\varphi(s))\Delta s+c(t-a),\quad
t\in[a,b]_\mathbb{T} \, .
$$
Using the boundary condition $y(b)=B$ we have
$$c=\frac{\int_a^b\ln(\varphi(t))\Delta t+B}{b-a}=C,$$
with $C$ as in \eqref{eq1}. A simple calculation shows that
${\mathcal F}_{\min}=(b-a)e^C$.
\end{proof}

\begin{thm}
\label{thm7} Let
$\varphi:[a,b]^\kappa_\mathbb{T}\rightarrow\mathbb{R}$ be a
positive and rd-continuous function. Then, among all
$C_{\textup{rd}}^1$-functions
$y:[a,b]_{\mathbb{T}}\rightarrow\mathbb{R}$ satisfying
$y^\Delta>0$, $y(a)=0$, and $y(b)=B$, with
\begin{equation}
\label{in5}
\frac{B+\int_a^b\varphi(s)\Delta
s}{b-a}>\varphi(t),\quad t\in[a,b]_\mathbb{T}^\kappa \, ,
\end{equation}
the functional
$${\mathcal F}(y(\cdot))
=\int_a^b[\varphi(t)+y^\Delta(t)]\ln[\varphi(t)+y^\Delta(t)]\Delta
t$$
has minimum value ${\mathcal F}_{\min}=(b-a)C\ln(C)$
attained when
$$y(t)=C(t-a)-\int_a^t\varphi(s)\Delta s,\quad
t\in[a,b]_\mathbb{T} \, ,
$$
where
\begin{equation}
\label{eq2}
C=\frac{B+\int_a^b\varphi(s)\Delta s}{b-a} \, .
\end{equation}
\end{thm}

\begin{proof}
By Corollary~\ref{cor5},
\begin{equation*}
\begin{split}
{\mathcal F}(y(\cdot)) &\geq
\int_a^b [\varphi(t)+y^\Delta(t)]\Delta
t\ln\left({\frac{1}{b-a}\int_a^b [\varphi(t)+y^\Delta(t)]\Delta
t}\right)\\
&= \left(\int_a^b \varphi(t)\Delta
t+B\right)\ln\left({\frac{\int_a^b \varphi(t)\Delta
t+B}{b-a}}\right)
\end{split}
\end{equation*}
with ${\mathcal F}(y(\cdot)) = \left(\int_a^b \varphi(t)\Delta
t+B\right)\ln\left({\frac{\int_a^b \varphi(t)\Delta
t+B}{b-a}}\right)$ if and only if
$$\varphi(t)+y^\Delta(t)=c \quad \mbox{for some}\
c\in\mathbb{R},\quad
t\in[a,b]^\kappa_\mathbb{T}.$$ Upon integration from $a$ to $t$
(note that $y(a)=0$),
$$y(t)=c(t-a)-\int_a^t\varphi(s)\Delta s,
\quad t\in[a,b]_\mathbb{T}.$$
Using the boundary condition $y(b)=B$, we have
$$c=\frac{B+\int_a^b\varphi(s)\Delta s}{b-a}=C,$$
where $C$ is as in \eqref{eq2}. Note that with this choice of $y$
we have, using \eqref{in5}, that $y^\Delta(t)=C-\varphi(t)>0$,
$t\in[a,b]^\kappa_\mathbb{T}$. It follows that
${\mathcal F}_{\min}=(b-a)C\ln(C)$.
\end{proof}


\section{An Example}
\label{sec:ex}

Let $\mathbb{T}=\mathbb{Z}$, $a=0$, $b=5$, $B=25$ and
$\varphi(t)=2t+1$ in Theorem~\ref{thm7}:
\begin{example}
The functional
$${\mathcal F}(y(\cdot))
=\sum_{t=0}^4[(2t+1)+(y(t+1)-y(t))]\ln[(2t+1)+(y(t+1)-y(t))],$$
defined for all $y:[0,5]\cap\mathbb{Z}\rightarrow\mathbb{R}$ such
that $y(t+1)>y(t)$ for all $t\in[0,4]\cap\mathbb{T}$, attains its
minimum when
$$y(t)=10t-t^2,\quad t\in[0,5]\cap\mathbb{Z},$$
and ${\mathcal F}_{\min}=50\ln(10)$.
\end{example}
\begin{proof}
First we note that $\max\{\varphi(t):t\in[0,4]\cap\mathbb{Z}\}=9$.
Hence
$$\frac{B+\sum_{k=0}^4\varphi(k)}{b-a}=\frac{25+25}{5}=10>9\geq\varphi(t).$$
Observing that since, when $\mathbb{T}=\mathbb{Z}$, $(t^2)^\Delta=2t+1$,
we just have to invoke Theorem~\ref{thm7} to get the desired
result.
\end{proof}

\begin{rem}
There is an inconsistency in \cite[Theorem~3.6]{wsc}
due to the fact that the bound on the functional $I$
considered there is not constant. For example,
let $a=A=1$, $\varphi(x)=x+1$, and $\tilde{y}(x)=x$ for all $x\in[0,1]$.
Then the hypotheses of \cite[Theorem~3.6]{wsc} are satisfied.
Moreover,
$$I(\tilde{y}(x))=\int_0^1\ln(\varphi(x)\tilde{y}'(x))dx
=\left[(x+1)(\ln(x+1)-1)\right]_{x=0}^{x=1}=2\ln(2)-1\approx
0.386.$$
According to \cite[Theorem~3.6]{wsc}, the maximum of the
functional $I$ is given by $I_{\max}=-\ln(C)$, where
$$C=\frac{1}{A}\int_0^1\frac{1}{\varphi(x)}dx.$$
A simple calculation shows that $C=\ln(2)$. Hence
$I_{\max}=-\ln(\ln(2))\approx 0.367$. Therefore,
$I(\tilde{y}(x))>I_{\max}$.
\end{rem}



\end{document}